\documentclass{amsart}
\title{Abstract matrix-tree theorem}
\author[Yu.\,Burman]{Yurii Burman}
\address{National Research University Higher School of Economics, 119048, 6
Usacheva str., Moscow, Russia, and Independent University of Moscow,
119002, 11 B.Vlassievsky per., Moscow, Russia}
\email{burman@mccme.ru}
\date{}

\makeatletter

\RequirePackage{amssymb}

\newcommand{\theoremName}{Theorem}
\newcommand{\lemmaName}{Lemma}
\newcommand{\corollaryName}{Corollary}
\newcommand{\statementName}{Proposition}

\newcommand{\remarkName}{Remark}
\newcommand{\exampleName}{Example}
\newcommand{\definitionName}{Definition}
\newcommand{\problemName}{Problem}
\newcommand{\proofName}{Proof}
\renewcommand{\proofname}{\proofName}
\newcommand{\answerName}{Answer}
\newcommand{\hintName}{Hint}

\theoremstyle{plain}

\newtheorem {theorem}{\theoremName}
\newtheorem {lemma}{\lemmaName}
\newtheorem {corollary}{\corollaryName}

\newtheorem {proposition}{\statementName}
\newtheorem {Theorem}{\theoremName}
\def \theTheorem {\!\!}

\theoremstyle{remark}

\newtheorem{Remark}{\remarkName}

\newtheorem{example}{\exampleName}

\theoremstyle{definition}

\newtheorem{definition}{\definitionName}

\let\@newpf\proof
\let\proof\relax

\def \namepf[#1] {\@newpf[\proofname\ #1]}
\newenvironment{proof}{\@ifnextchar[{\namepf}{\@newpf[\proofname]}}{\qed\endtrivlist}

\newcounter{qst}
\@addtoreset{qst}{problem}

\def \Integer {{\mathbb Z}}

\def \Complex {{\mathbb C}}

\def \lnorm#1\rnorm {\vphantom{#1}\left\|\smash{#1}\right\|}
\def \lmod#1\rmod {\vphantom{#1}\left|\smash{#1}\right|}

\newcommand \bydef {\stackrel{\mbox{\scriptsize def}}{=}}

\@ifundefined{name}{\newcommand \name[1] {\mathop{\rm #1}\nolimits}}%
{\relax}

\renewcommand \phi {\varphi}
\renewcommand \rho {\varrho}
\renewcommand \emptyset {\varnothing}

\makeatother

\usepackage[mathscr]{euscript}

\newcommand{\DT}[1]{#1 \dots #1}
\def \Graph {\mathcal G}
\def \Undir {\mathcal Y}
\def \POS {\mathcal P}
\def \Lapl {\Delta}
\def \AC {\name{\mathfrak A}}
\def \SSC {\name{\mathfrak S}}
\let \Det=\det
\def \det {\Det\nolimits}
\def \Manif {\mathcal M}
\def \SumSub {\name{\mathscr S}}

\numberwithin{equation}{section}

\theoremstyle{plain}
\newtheorem{Conjecture}{Conjecture}

\numberwithin{theorem}{section}

\newtheorem{lemma}[theorem]{Lemma}

\newtheorem{proposition}[theorem]{Proposition}

\newtheorem{corollary}[theorem]{Corollary}

\theoremstyle{remark}

\newtheorem{example}[theorem]{Example}

\theoremstyle{definition}

\newtheorem{definition}[theorem]{Definition}

\makeatother

\begin{document}

 \begin{abstract}
The classical matrix-tree theorem discovered by G.\,Kirchhoff in 1847
relates the principal minor of the $n \times n$ Laplace matrix to a
particular sum of monomials of matrix elements indexed by directed trees
with $n$ vertices and a single sink. In this paper we consider a
generalization of this statement: for any $k \ge n$ we define a degree $k$
polynomial $\det_{n,k}$ of matrix elements and prove that this polynomial
applied to the Laplace matrix gives a sum of monomials indexed by acyclic
graphs with $n$ vertices and $k$ edges.
 \end{abstract}

\maketitle

\section{Introduction and the main results}

\subsection{Principal definitions}\label{SSec:Def}

Denote by $\Gamma_{n,k}$ the set of all directed graphs with $n$ vertices
numbered $1 \DT, n$ and $k$ edges numbered $1 \DT, k$. We will write $e =
[ab]$ if $e$ is an edge from vertex $a$ to vertex $b$; in particular,
$[aa]$ means a loop attached to the vertex $a$. We will treat elements of
$\Gamma_{n,k}$ as sequences of edges: $G = (e_1 \DT, e_k) \in \Gamma_{n,k}$
means a graph where the edge $e_\ell$ has number $\ell$, for all $\ell = 1
\DT, k$. By a slight abuse of notation $e \in G$ will mean that $e$ is an
edge of $G$ (regardless of number).

Let $G \in \Gamma_{n,k}$ and $e \in G$. By $G \setminus e$, $G/e$ and
$G_e^{\vee}$ we will denote the graph $G$ with $e$ deleted, $e$ contracted
and $e$ reversed, respectively. Note for correctness that since $G
\setminus e \in \Gamma_{n,k-1}$, one has to change the edge numbering in
$G$ after deleting $e$: namely, if $e$ bears number $s$ in $G$ then the
numbers of the edges are preserved if they are less than $s$ and lowered by
$1$ otherwise. For $G/e \in \Gamma_{n-1,k-1}$ the same renumbering is
applied both to the edges and to the vertices. The contracted edge $e$
should not be a loop.

A graph $H \in \Gamma_{n,m}$ is called a subgraph of $G \in \Gamma_{n,k}$
(notation $H \subseteq G$) if $H$ is obtained from $G$ by deletion of
several (possibly zero) edges.

Denote by $\Graph_{n,k}$ a vector space over $\Complex$ spanned by
$\Gamma_{n,k}$. The direct sum $\Graph_n \bydef \bigoplus_{k=0}^\infty
\Graph_{n,k}$ bears the structure of an associative algebra: one defines a
product of the graphs $G_1 = (e_1 \DT, e_{k_1}) \in \Gamma_{n,k_1}$ and
$G_2 = (h_1 \DT, h_{k_2}) \in \Gamma_{n,k_2}$ as $G_1*G_2 \bydef (e_1 \DT,
e_{k_1}, h_1 \DT, h_{k_2}) \in \Gamma_{n,k_1+k_2}$; then $*$ is extended to
the whole $\Graph_n$ as a bilinear operation. Note that $G_1*G_2 \ne
G_2*G_1$ (the edges are the same but the edge numbering is different), so
the algebra $\Graph_n$ is not commutative.

We call a graph $G \in \Gamma_{n,k}$ {\em strongly connected} if every two
its vertices can be joined by a directed path. A graph is {\em strongly
semiconnected} if every its connected component (in the topological sense)
is strongly connected; equivalently, if every its edge is a part of a
directed cycle. A strongly semiconnected graph may contain isolated
vertices (i.e.\ vertices not incident to any edge); by $\SSC_{n,k}^{\{i_1
\DT, i_s\}}$ we denote the set of strongly semiconnected graphs $G \in
\Gamma_{n,k}$ such that the vertices $i_1 \DT, i_s$, and only they, are
isolated. By $\SSC_{n,k} \bydef \bigcup_{I \subset \{1 \DT, n\}}
\SSC_{n,k}^I$ we will denote the set of all strongly semiconnected graphs.

We call a graph $G \in \Gamma_{n,k}$ {\em acyclic} if it contains no
directed cycles. Recall that a vertex $a$ of the graph $G$ is called a {\em
sink} if $G$ has no edges starting from $a$. Note that an isolated vertex
is a sink but a vertex with a loop attached to it is not. We denote by
$\AC_{n,k}^{\{i_1 \DT, i_s\}}$ the set of acyclic graphs $G \in
\Gamma_{n,k}$ such that the vertices $i_1 \DT, i_s$, and only they, are
sinks. By $\AC_{n,k} \bydef \bigcup_{I \subset \{1 \DT, n\}} \AC_{n,k}^I$
we will denote the set of all acyclic graphs.

 \begin{example}\label{Ex:SSC}
If a vertex of a strongly semiconnected graph $G \in \SSC_{n,k}^I$ is not
isolated then there is at least one edge starting from it; so if $I = \{i_1
\DT, i_s\}$ and $\SSC_{n,k}^I \ne \emptyset$ then $k \ge n-s$.

Let $k=n-s$. If $G \in \Gamma_{n,k}^I$ then for any vertex $i \notin I$
there is exactly one edge $[i, \sigma(i)]$ starting at it and exactly one
edge $[j, \sigma(j)] = [j,i]$ finishing at it (that is, $\sigma(j) = i$).
Hence $\sigma$ is a bijection $\{1 \DT, n\} \setminus I \to \{1 \DT, n\}
\setminus I$ (a permutation of $k = n-s$ points).

Geometrically $G$ is a union of disjoint directed cycles passing through
all vertices except $i_1 \DT, i_s$.
 \end{example}

 \begin{example}\label{Ex:Forest}
Let $n > k$; then any graph $G \in \Gamma_{n,k}$ contains at least $n-k$
connected components. If $G$ is acyclic then every its connected component
contains a sink. So for $I = \{i_1 \DT, i_s\}$ if $\AC_{n,k}^I \ne
\emptyset$ then $k \ge n-s$.

Let $k = n-s$. Then the elements of $\AC_{n,k}^I$ are forests of $s$
components, each component containing exactly one vertex $i_\ell \in I$
(for some $\ell = 1 \DT, s$), which is its only sink. This component is a
tree and every its edge is directed towards the sink $i_\ell$.
 \end{example}

\subsection{Determinants and minors}

Let $W = (w_{ij})$ be a $n \times n$-matrix; denote by $\langle W \vert:
\Graph_{n,k} \to \Complex$ a linear functional acting on the basic element
$G \in \Gamma_{n,k}$ as
 \begin{equation*}
\langle W \mid G\rangle \bydef \prod_{[ij] \in G} w_{ij}.
 \end{equation*}
Note that $\langle W \mid G\rangle$ is independent of the edge numbering in
$G$; in particular, $\langle W \mid G_1*G_2 - G_2*G_1\rangle = 0$ for all
$G_1, G_2$.

For a function $f: \bigcup_s \Gamma_{n,s} \to \Complex$ and a graph $G \in
\Gamma_{n,k}$ introduce the notation
 \begin{equation}\label{Eq:SumSubgr}
\SumSub(f;G) \bydef \sum_{H \subseteq G} f(H).
 \end{equation}
For a set of graphs $\mathfrak B \subset \Gamma_{n,k}$ denote
 \begin{align*}
&U(\mathfrak B) \bydef \sum_{G \in \mathfrak B} G \in \Graph_{n,k}, \\
&X(\mathfrak B) \bydef \sum_{G \in \mathfrak B} (-1)^{\beta_0(G)} G \in
\Graph_{n,k};
 \end{align*}
$\beta_0(G)$ here means the $0$-th Betti number of $G$, i.e.\ the number of
its connected components (in the topological sense).

 \begin{definition} \label{Df:Minors}
The element
 \begin{equation*}
\det_{n,k}^I \bydef \frac{(-1)^k}{k!} X(\SSC_{n,k}^I) \in
\Graph_{n,k}
 \end{equation*}
is called a {\em universal diagonal $I$-minor} of degree $k$; in
particular, $\det_{n,k}^\emptyset$ is called a {\em universal determinant}
of degree $k$.

The element
 \begin{equation*}
\det_{n,k}^{i/j} \bydef \frac{(-1)^k}{k!} X(\{G \in \Graph_{n,k}
\mid ([ij])*G \in \SSC_{n,k+1}^\emptyset\})
 \end{equation*}
is called a {\em universal (codimension $1$) $(i,j)$-minor} of degree $k$.
 \end{definition}

 \begin{example}\label{Ex:Det}
Example \ref{Ex:SSC} implies that if $I = \{i_1 \DT, i_s\}$ and $k < n-s$
then $\det_{n,k}^I = 0$.

Let $k = n$ and $I = \emptyset$. By Example \ref{Ex:SSC} the graphs $G \in
\SSC_{n,n}^\emptyset$ are in one-to-one correspondence with permutations
$\sigma$ of $\{1 \DT, n\}$. It is easy to see that $(-1)^{\beta_0(G)}$ is
equal to $(-1)^n$ if $\sigma$ is even and to $-(-1)^n$ if it is odd.
Geometrically $G$ is a union of disjoint directed cycles. If the order of
vertices in all the cycles is fixed, then there are $n!$ ways to assign
numbers $\{1 \DT, n\}$ to the edges; this implies the equality
 \begin{equation*}
\langle W \mid \det_{n,n}^\emptyset\rangle = \sum_\sigma
(-1)^{\text{parity of $\sigma$}} w_{1\sigma(1)} \dots w_{n \sigma(n)} =
\det W
 \end{equation*}
for any matrix $W = (w_{ij})$. Similarly, for any set $I = \{i_1 \DT,
i_s\}$ the value $\nolinebreak[0]\langle W \mid \det_{n,n-s}^I\rangle$ is
equal to the diagonal minor of the matrix $W$ obtained by deletion of the
rows and the columns $i_1 \DT, i_s$. Also $\langle W \mid
\det_{n,n-1}^{i/j}\rangle$ is equal to the codimension $1$ minor of $W$
obtained by deletion of the row $i$ and the column $j$. This explains the
terminology of Definition \ref{Df:Minors}.
 \end{example}

The elements $\det_{n,k}^I$ exhibit some properties one would expect from
determinants and minors:

 \begin{proposition} \label{Pp:DetProp}\strut
 \begin{enumerate}
\item\label{It:RowCol} (generalized row and column expansion)
 \begin{equation}\label{Eq:DetViaMinors}
\det_{n,k}^\emptyset = \frac{1}{k} \sum_{i,j=1}^n ([ij]) *
\det_{n,k-1}^{i/j}.
 \end{equation}

\item\label{It:PDer} (partial derivative with respect to a diagonal matrix
element) Let matrix elements $w_{ij}$, $i,j = 1 \DT, n$, of the matrix $W$
be independent (commuting) variables. Then for any $i = 1 \DT, n$ and any
$m = 1 \DT, k$ one has
 \begin{equation}\label{Eq:Deriv}
\frac{\partial^m}{\partial w_{ii}^m} \langle W \mid
\det_{n,k}^\emptyset\rangle = \langle W \mid
\det_{n,k-m}^\emptyset + \det^{\{i\}}_{n,k-m}\rangle.
 \end{equation}
 \end{enumerate}
 \end{proposition}

See \cite[Lemma 86]{Epstein} for a formula similar to \eqref{Eq:Deriv}
(with $m=1$ and a finite difference instead of a derivative).

\subsection{Main results}\label{SSec:Main}

Let $G \in \Gamma_{n,k}$, $p \in \{1 \DT, k\}$ and $i,j \in \{1 \DT, n\}$.
Denote by $R_{ab;p} G \in \Gamma_{n,k}$ the graph obtained from $G$ by
replacement of its $p$-th edge by the edge $[ab]$ bearing the same number
$p$.

Consider now a linear operator $B_p: \Graph_{n,k} \to \Graph_{n,k}$
acting on every basic element $G \in \Gamma_{n,k}$ as follows:
 \begin{equation*}
B_p(G) =  \begin{cases}
G, &\text{if the $p$-th edge of $G$ is not a loop},\\
-\sum_{b \ne a} R_{ab;p} G, &\text{if the $p$-th edge of $G$ is the loop
$[aa]$}.
 \end{cases}
 \end{equation*}
In particular, $B_p = 0$ if $n=1$ (and $k > 0$).

 \begin{definition} \label{Df:Laplace}
The product $\Lapl \bydef B_1 \dots B_k: \Graph_{n,k} \to \Graph_{n,k}$ is
called {\em Laplace operator}.
 \end{definition}

If $n = 1$ and $k > 0$ then $\Delta = 0$; also take $\Delta = \name{id}$ by
definition if $k = 0$.

 \begin{Remark}
The operators $B_p$, $p = 1 \DT, k$, are commuting idempotents: $B_p^2 =
B_p$ and $B_p B_q = B_q B_p$ for all $p, q = 1 \DT, k$. Therefore, $\Lapl$
is an idempotent, too: $\Lapl^2 = \Lapl$.
 \end{Remark}

Let $W = (w_{ij})_{i,j=1}^n$ be a $n \times n$-matrix, like in Example
\ref{Ex:Det} and Proposition \ref{Pp:DetProp}. Denote by $\widehat{W}$ the
corresponding Laplace matrix, i.e.\ a matrix with nondiagonal elements
$w_{ij}$ ($1 \le i \ne j \le n$) and diagonal elements $-\sum_{j \ne i}
w_{ij}$ ($1 \le i \le n$). It follows from Definition \ref{Df:Laplace} that
 \begin{equation*}
\langle \widehat{W} \mid X \rangle = \langle W \mid \Lapl(X) \rangle
 \end{equation*}
for any $X \in \Graph_{n,k}$. This equation explains the name ``Laplace
operator'' for $\Lapl$. Note that since $\Lapl(X)$ is a sum of graphs
containing no loops, one is free to change diagonal entries of $W$ in the
right-hand side; in particular, one can use $\widehat{W}$ instead.

The main results of this paper are the following two theorems:

 \begin{theorem}[abstract matrix-tree theorem for diagonal minors]\label{Th:Diag}
 \begin{equation}\label{Eq:LaplDiag}
\Lapl(\det_{n,k}^I) = \frac{(-1)^n}{k!} U(\AC_{n,k}^I).
 \end{equation}
 \end{theorem}

\noindent and

 \begin{theorem}[abstract matrix-tree theorem for codimension $1$ minors]\label{Th:Codim1}
 \begin{equation}\label{Eq:LaplCodim1}
\Lapl(\det_{n,k}^{i/j}) = \frac{(-1)^n}{k!} U(\AC_{n,k}^{\{i\}}).
 \end{equation}
 \end{theorem}

Applying the functional $\langle \widehat{W}\vert\relax$ to equation
\eqref{Eq:LaplDiag} with $k = n-s$ and to equation \eqref{Eq:LaplCodim1}
with $k = n-1$ and using Examples \ref{Ex:Det} and \ref{Ex:Forest} one
obtains

 \begin{corollary}\label{Cr:Diag}
The diagonal minor of the Laplace matrix obtained by deletion of the rows
and columns numbered $i_1 \DT, i_s$ is equal to
$\frac{(-1)^{n-s}}{(n-s)!}\langle W \mid U(\AC_{n,n-s}^I)\rangle$, that is,
to $(-1)^{n-s}$ times the sum of monomials $w_{a_1 b_1} \dots w_{a_{n-s}
b_{n-s}}$ such that the graph $([a_1 b_1] \DT, [a_{n-s} b_{n-s}])$ is a
$s$-component forest where every component contains exactly one vertex
$i_\ell$ for some $\ell = 1 \DT, s$, and all the edges of the component are
directed towards $i_\ell$.
 \end{corollary}

\noindent and

 \begin{corollary}\label{Cr:Codim1}
The minor of the Laplace matrix obtained by deletion of its $i$-th row and
its $j$-th column is equal to $(-1)^{n-1}$ times the sum of monomials
$w_{a_1 b_1} \dots w_{a_{n-1} b_{n-1}}$ such that the graph $([a_1 b_1]
\DT, [a_{n-1} b_{n-1}])$ is a tree with all the edges directed towards the
vertex $i$.
 \end{corollary}

Corollaries \ref{Cr:Diag} and \ref{Cr:Codim1} are particular cases of the
celebrated matrix-tree theorem first discovered by G.\,Kirchhoff
\cite{Kirch} in 1847 (for symmetric matrices and diagonal minors of
codimension $1$) and proved in its present form by W.\,Tutte
\cite{TutteMTT}.

Consider now the following functions on $\Gamma_{n,k}$:
 \begin{align*}
\sigma(G) &= \begin{cases}
(-1)^{\beta_1(G)}, &G \in \Gamma_{n,k} \text{ is strongly semiconnected},\\
0 &\text{otherwise},
 \end{cases}\\
\text{and\hspace{3.5cm}}\\
\alpha(G) &=  \begin{cases}
(-1)^k, &G \in \Gamma_{n,k} \text{ is acyclic},\\
0 &\text{otherwise}.
 \end{cases}
 \end{align*}

Theorem \ref{Th:Diag} follows from the two equivalent statements (see
Section \ref{Sec:Proofs} for details):

 \begin{theorem}\label{Th:Direct}
$\SumSub(\alpha;G) = (-1)^k \sigma(G)$ for $G \in \Gamma_{n,k}$.
 \end{theorem}

{\def \theTheorem {\ref{Th:Direct}'}
 \begin{Theorem}
$\SumSub(\sigma;G) = (-1)^k\alpha(G)$ for $G \in \Gamma_{n,k}$.
 \end{Theorem}
}
\noindent Here $\SumSub$ is summation over subgraphs, as defined by 
\eqref{Eq:SumSubgr}. These theorems are essentially \cite[Proposition 
6.16]{Bernardi}. We will nevertheless give their proofs in Section 
\ref{Sec:Proofs} thus answering a request for a direct proof expressed in 
\cite{Bernardi} (the original proof in \cite{Bernardi} is a specialization 
of a much more general identity).

\subsection{A digression: undirected graphs and the universal Potts
partition function}

Denote by $\Upsilon_{n,k}$ the set of all {\em undirected} graphs with $n$
vertices numbered $1 \DT, n$ and $k$ edges numbered $1 \DT, k$. Denote by
$\lmod \,\cdot\,\rmod: \Gamma_{n,k} \to \Upsilon_{n,k}$ the ``forgetful''
map replacing every edge by its undirected version; the edge numbering is
preserved. By $\Undir_{n,k}$ denote a vector space spanned by
$\Upsilon_{n,k}$; then $\lmod \,\cdot\,\rmod$ is extended to the linear map
$\Graph_{n,k} \to \Undir_{n,k}$. The notion of a subgraph and the notation
$\SumSub$ (see \eqref{Eq:SumSubgr}) for undirected graphs are similar to
those for $\Graph_{n,k}$. One can also define the operators $B_p:
\Undir_{n,k} \to \Undir_{n,k}$, $p = 1 \DT, k$, and the Laplace operator
$\Delta: \Undir_{n,k} \to \Undir_{n,k}$ for undirected graphs exactly as in
Definition \ref{Df:Laplace}.

For any $G \in \Undir_{n,k}$ consider the two-variable polynomial:
 \begin{equation}\label{Eq:DefPotts}
Z_G(q,v) = \SumSub(q^{\beta_0(H)} v^{\#\text{of edges of $H$}};G).
 \end{equation}
called {\em Potts partition function}. It is related \cite[Eq.\
(2.26)]{Sokal} to the Tutte polynomial $T_G$ of the graph $G$ as
 \begin{equation*}
T_G(x,y) = (x-1)^{-\beta_0(G)} (y-1)^{-n} Z((x-1)(y-1), y-1;G).
 \end{equation*}
Values of $Z_G$ in some points have a special combinatorial interpretation,
in particular

 \begin{proposition}[\protect{\cite[V, (8)
and (10)]{WelshMerino}}]\label{Pp:SpecVal}\strut
 \begin{align*}
Z_G(-1,1) &= (-1)^{\beta_0(G)} 2^{\#\text{\upshape of loops of $G$}} \#
\{\Phi \in \SSC_{n,k} \subset \Gamma_{n,k} \mid \lmod \Phi\rmod = G\}.\\
Z_G(-1,-1) &= (-1)^n \# \{\Phi \in \AC_{n,k} \subset \Gamma_{n,k} \mid
\lmod \Phi\rmod = G\}.
 \end{align*}
 \end{proposition}
\noindent (Recall that by $\SSC_{n,k}$ and $\AC_{n,k}$ we denote the sets
of all strongly semiconnected and acyclic graphs in $\Gamma_{n,k}$,
respectively.)

 \begin{corollary} \label{Cr:NumSSC}
For any graph $G \in \Upsilon_{n,k}$ one has
 \begin{equation*}
\# \{\Phi \in \SSC_{n,k} \subset \Gamma_{n,k} \mid \lmod \Phi\rmod = G\} =
(-1)^{\beta_0(G)} Z_{\widehat G}(-1,1)
 \end{equation*}
where $\widehat G$ is the graph $G$ with all the loops deleted.
 \end{corollary}

 \begin{proof}
The definition \eqref{Eq:DefPotts} of the Potts partition function implies
immediately that $Z_G(q,v) = (v+1)^{\#\text{of loops of $G$}} Z_{\widehat
G}(q,v)$.
 \end{proof}

Consider now the {\em universal Potts partition function}
 \begin{equation*}
{\mathcal Z}_{n,k}(q,v) \bydef \sum_{G \in \Upsilon_{n,k}} Z_G(q,v)G \in
\Undir_{n,k}
 \end{equation*}
and its ``shaved'' version
 \begin{equation*}
\widehat{{\mathcal Z}}_{n,k}(q,v) \bydef \sum_{G \in \Upsilon_{n,k}}
Z_{\widehat G}(q,v)G \in \Undir_{n,k}.
 \end{equation*}

 \begin{proposition} \label{Pp:LaplTutte}
 \begin{equation*}
\Delta\widehat{{\mathcal Z}}_{n,k}(-1,1) = (-1)^k {\mathcal
Z}_{n,k}(-1,-1).
 \end{equation*}
 \end{proposition}
\noindent Note that by Proposition \ref{Pp:SpecVal} the right-hand side of
the equality contains only graphs without loops, as does the left-hand
side.

 \begin{proof}
Corollary \ref{Cr:NumSSC} implies that
 \begin{equation*}
\widehat{{\mathcal Z}}_{n,k}(-1,1) = (-1)^k k!\sum_{I \subset \{1 \DT, n\}}
\lmod \det_{n,k}^I\rmod.
 \end{equation*}
Apply now the Laplace operator $\Delta$ to both sides of the equality.
Apparently, $\Delta$ commutes with the forgetful map: $\lmod \Delta(x)\rmod
= \Delta(\lmod x\rmod)$ for any $x \in \Graph_{n,k}$. Therefore by Theorem
\ref{Th:Diag} and Proposition \ref{Pp:SpecVal}
 \begin{align*}
\Delta\widehat{{\mathcal Z}}_{n,k}(-1,1) &= (-1)^k k!\sum_{I \subset \{1
\DT, n\}} \lmod \Delta \det_{n,k}^I\rmod = (-1)^{n+k} \sum_{I \subset \{1
\DT, n\}} \lmod U(\AC_{n,k}^I)\rmod \\
&= (-1)^k {\mathcal Z}_{n,k}(-1,-1).
 \end{align*}
 \end{proof}

Proposition \ref{Pp:LaplTutte} admits several generalizations. The author
is planning to write a separate paper considering action of the Laplace
operator on the universal Potts functions and their oriented-graph
versions.

\subsection{An application: invariants of $3$-manifolds}

Universal determinants have an application in $3$-dimensional topology, due
to M.\,Polyak. We describe it briefly here; see \cite{Toronto} and the
MSc.\ thesis \cite{Epstein} for detailed definitions, formulations and
proofs.

A {\em chainmail graph} is defined as a planar graph, possibly with loops
but without parallel edges; the edges (including loops) are supplied with
integer weights. We denote by $w_{ij} = w_{ji}$ the weight of the edge
joining vertices $i$ and $j$; $w_{ii}$ is the weight of the loop attached
to the vertex $i$. If the edge $[ij]$ is missing then $w_{ij} = 0$ by
definition.

There is a way (see \cite{Toronto}) to define for every chainmail graph $G$
a closed oriented $3$-manifold $\Manif(G)$; any closed oriented
$3$-manifold is $\Manif(G)$ for some $G$ (which is not unique). To the
chainmail graph $G$ with $n$ vertices one associates two $n \times
n$-matrices: the adjacency matrix $W(G) = (w_{ij})$ and the Laplace (better
to say, Schroedinger) matrix $L(G) = (l_{ij})$ where $l_{ij} \bydef w_{ij}$
for $i \ne j$ and $l_{ii} \bydef w_{ii} - \sum_{j \ne i} w_{ij}$. If all
$w_{ii} = 0$ (such $G$ is called a balanced graph) then $L(G)$ is the usual
(symmetric, degenerate) Laplace matrix $\widehat{W}$ from Section
\ref{SSec:Main}.

 \begin{Theorem}[\cite{Toronto}; see details of the proof in
\cite{Epstein}]\strut
 \begin{enumerate}
\item The rank of the homology group $H_1(\Manif(G),\Integer)$ is equal to
$\dim \name{Ker} L(G)$.

\item If $L(G)$ is nondegenerate (so that $\Manif(G)$ is a rational
homology sphere and $H_1(\Manif(G),\Integer)$ is finite) then
 \begin{equation}\label{Eq:3ManifDet}
\lmod H_1(\Manif(G),\Integer)\rmod = \lmod \det L(G)\rmod = \lmod\langle
L(G) \mid \det_{n,n}^\emptyset\rangle\rmod.
 \end{equation}

\item If $L(G)$ is nondegenerate then
 \begin{equation}\label{Eq:3ManifTheta}
\langle W(G) \mid \Theta_n\rangle = 12 \det L(G)
\bigl(\lambda_{CW}(\Manif(G)) - \frac{1}{4} \name{sign}(L(G))\bigr)
 \end{equation}
where $\lambda_{CW}$ is the Casson--Walker invariant \cite{Walker} of the
raional homology sphere $\Manif(G)$, $\name{sign}$ is the signature of the
symmetric matrix $L(G)$, and $\Theta_n$ is an element of $\Graph_{n,n+1}
\oplus \Graph_{n,n-1}$ defined as
 \begin{equation*}
\Theta_n \stackrel{\mathrm{def}}{=}\det_{n,n+1}^\emptyset - \sum_{1 \le i
\ne j \le n} ([ij])*\det_{n,n-2}^{\{i,j\}} - \sum_{i=1}^n
\det_{n,n-1}^{\{i\}}.
 \end{equation*}
 \end{enumerate}
 \end{Theorem}

Conjecturally, \eqref{Eq:3ManifDet} and \eqref{Eq:3ManifTheta} begin a
series of formulas for invariants of $3$-manifolds. See \cite{Toronto} for
details.

Applying $\Delta$ to the element $\Theta_n$ and using Theorem \ref{Th:Diag}
and Corollary \ref{Cr:Diag} one obtains

 \begin{corollary}
$\Delta \Theta_n = -2 U(\AC_{n,n-1})$. Therefore if $G$ is balanced then
$\langle L(G) \mid \Theta_n\rangle$ is equal to $-2$ times the codimension
$1$ diagonal minor of $L(G)$.
 \end{corollary}

The last assertion is \cite[Theorem 84]{Epstein}.

\subsection*{Acknowledgements}
The research was inspired by numerous discussions with prof.\ Michael
Polyak (Haifa Technion, Israel) whom the author wishes to express his most
sincere gratitude.

The research was funded by the Russian Academic Excellence Project `5-100'
and by the grant No.~15-01-0031 ``Hurwitz numbers and graph isomorphism''
of the Scientific Fund of the Higher School of Economics.

\section{Proofs}\label{Sec:Proofs}

We start with proving Proposition \ref{Pp:DetProp} (Section
\ref{SSec:PrDetProp}), to continue with Theorems \ref{Th:Direct} and
\ref{Th:Direct}' (Sections \ref{SSec:Equiv} and \ref{SSec:PrDirect}).
Theorem \ref{Th:Diag} will then follow from Theorem \ref{Th:Direct}'
(Section \ref{SSec:DiagFromDirect}), and Theorem \ref{Th:Codim1}, from
Theorem \ref{Th:Diag} and assertion \ref{It:RowCol} of Proposition
\ref{Pp:DetProp} (Section \ref{SSec:PrCodim1}).

For two vertices $a, b \in G \in \Gamma_{n,k}$ we will write $a \succeq b$
if $G$ contains a directed path starting at $a$ and finishing at $b$; also
$a \succeq a$ for any $a$ by definition.

\subsection{Proof of Proposition \ref{Pp:DetProp}}\label{SSec:PrDetProp}

 \begin{proof}[of assertion \ref{It:RowCol}]
Let $G$ be a strongly semiconnected graph, and $[ij] \in G$ be its edge
carrying number $1$. Since $G$ is strongly semiconnected, $j \succeq i$ in
$G \setminus ([ij])$, and therefore $\beta_0(G \setminus ([ij])) =
\beta_0(G)$. Then $G$ enters the left-hand side and the $(i,j)$-th term of
the sum at the right-hand side of \eqref{Eq:DetViaMinors} with the same
coefficient.
 \end{proof}

 \begin{proof}[of assertion \ref{It:PDer}]
Denote by $\SSC_{n,k}^{[i:q]}$ the set of all graphs $G \in
\SSC_{n,k}^\emptyset$ having $q$ loops ($0 \le q \le k$) attached to vertex
$i$. The graph $\hat G$ obtained from $G$ by deletion of
all these loops (with the relevant renumbering of the remaining edges)
belongs either to $\SSC_{n,k-q}^{[i:0]} \subset \SSC_{n,k-q}^\emptyset$ or,
if $q > 0$, to $\SSC_{n,k-q}^{\{i\}}$. Vice versa, if $q > 0$ and $\hat G
\in \SSC_{n,k-q}^{[i:0]} \cup \SSC_{n,k-q}^{\{i\}}$ then $G \in
\SSC_{n,k}^{[i:q]}$. Deletion of a loop does not break a graph, so
$\beta_0(G) = \beta_0(\hat G)$.

If $G \in \SSC_{n,k}^{[i:q]}$ then there are $\binom{k}{q}$ ways to assign
numbers to the loops of $G$ attached to $i$. Since $\langle W \mid
G\rangle$ does not depend on the edge numbering, one has for $q > 0$
 \begin{equation*}
\langle W \mid X(\SSC_{n,k}^{[i:q]})\rangle = \binom{k}{q} w_{ii}^q
\langle W \mid X(\SSC_{n,k-q}^{[i:0]}) + X(\SSC_{n,k-q}^{\{i\}})\rangle,
 \end{equation*}
so that
 \begin{align}
\langle W \mid \det_{n,k}^\emptyset\rangle &= \frac{1}{k!}\sum_{q=0}^k
\langle W \mid X(\SSC_{n,k}^{[i:q]})\rangle\nonumber \\
&= \frac{1}{k!} \langle W \mid X(\SSC_{n,k}^{[i:0]})\rangle + \sum_{q=1}^k
\frac{w_{ii}^q}{q!(k-q)!} \langle W \mid X(\SSC_{n,k-q}^{[i:0]}) +
X(\SSC_{n,k-q}^{\{i\}})\rangle\nonumber \\
&= \sum_{q=0}^k \frac{w_{ii}^q}{q!(k-q)!} \langle W \mid
X(\SSC_{n,k-q}^{[i:0]}) + X(\SSC_{n,k-q}^{\{i\}})\rangle - \langle W \mid
\det_{n,k}^{\{i\}}\rangle.\label{Eq:Develop}
 \end{align}
The expressions $\langle W \mid X(\SSC_{n,k-q}^{[i:0]}) +
X(\SSC_{n,k-q}^{\{i\}})\rangle$ and $\langle W \mid
\det_{n,k}^{\{i\}}\rangle$ do not contain $w_{ii}$. So, applying the
operator $\frac{\partial^m}{\partial w_{ii}^m}$ to equation
\eqref{Eq:Develop} and using the equation again with $k-m$ in place of $k$
one gets \eqref{Eq:Deriv}.
 \end{proof}

\subsection{Theorems \ref{Th:Direct} and \ref{Th:Direct}' are
equivalent.}\label{SSec:Equiv}

Denote by $E(G)$ the set of edges of the graph $G \in \Gamma_{n,k}$. The
functions $\alpha$ and $\sigma$ do not depend on the edge
numbering; so the summation in the left-hand side of both theorems is
performed over the set $2^{E(G)}$ of subsets of $E(G)$. The
equivalence of the theorems is now a particular case of the Moebius
inversion formula \cite{Rota}. Namely, for any finite set $X$ the Moebius
function of the set $2^X$ partially ordered by inclusion is $\mu(S,T) =
(-1)^{\#(S \setminus T)}$, where $S, T \subseteq X$. Therefore one has
 \begin{align*}
\SumSub(&\sigma;G) = (-1)^k \alpha(G) \\
&\Longleftrightarrow \SumSub(\mu(G,H) (-1)^{\#\text{edges of H}}
\alpha(H);G) = \sigma(G) \\
&\Longleftrightarrow \SumSub((-1)^{k-\#\text{edges of H}}
(-1)^{\#\text{edges of H}}\alpha(H);G) = \sigma(G) \\
&\Longleftrightarrow \SumSub(\alpha(H);G) = (-1)^k \sigma(G).
 \end{align*}

\subsection{Proof of Theorem \ref{Th:Direct}}\label{SSec:PrDirect}

To prove the theorem we use simultaneous induction by the number of
vertices and the number of edges of the graph $G$. If $\mathcal R$ is some
set of subgraphs of $G$ (different in different cases) and $\chi_{\mathcal
R}$ is the characteristic function of this set then for convenience we will
write $\SumSub(f,\mathcal R) \bydef \SumSub(f \chi_{\mathcal R},G) = \sum_{H
\in \mathcal R} f(H)$ for any function $f$ on the set of subgraphs.

Consider now the following cases:

\subsubsection{$G$ is disconnected.}

Let $G = G_1 \DT\sqcup G_m$ where $G_i$ are connected components. A
subgraph $H \subset G$ is acyclic if and only if the intersection $H_i
\bydef H \cap G_i$ is acyclic for all $i$. Hence $\alpha(H) = \alpha(H_1)
\dots \alpha(H_m)$, and therefore $\SumSub(\alpha,G) = \SumSub(\alpha, G_1)
\dots \SumSub(\alpha, G_m)$. By the induction hypothesis $\SumSub(\alpha,
G_i) = (-1)^{k_i}\sigma(G_i)$ where $k_i$ is the number of edges of $G_i$.
So
 \begin{equation*}
\SumSub(\alpha,G) = \SumSub(\alpha, G_1) \dots \SumSub(\alpha, G_m) =
(-1)^{k_1 \DT+ k_m} \sigma(G_1) \dots \sigma(G_m) = (-1)^k \sigma(G).
 \end{equation*}
Now it will suffice to prove Theorem \ref{Th:Direct} for connected graphs
$G$.

\subsubsection{$G$ is connected and not strongly
connected.}\label{SSec:NotSSC}

In this case $G$ contains an edge $e$ which is not contained in any
directed cycle. For such $e$ if $H \subset G$ is acyclic and $e \notin H$
then $H \cup \{e\}$ is acyclic, too. The converse is true for any $e$: if
an acyclic $H \subset G$ contains $e$ then $H \setminus \{e\}$ is acyclic.
Therefore
 \begin{equation*}
\SumSub(\alpha,G) = \sum_{\substack{H \subset G \setminus \{e\},\\ H \text{
is acyclic}}} (-1)^{\#\text{edges of $H$}} + (-1)^{\#\text{edges of $H \cup
\{e\}$}} = 0 = \sigma(G).
 \end{equation*}
So it will suffice to prove Theorem \ref{Th:Direct} for strongly connected
graphs $G$.

\subsubsection{$G$ is strongly connected and contains a crucial
edge.}\label{SSec:SSCCrucial}

We call an edge $e$ of a strongly connected graph $G$ crucial if $G
\setminus \{e\}$ is not strongly connected. Suppose $e = [ab] \in G$ is a
crucial edge.

Denote by $\mathcal R_e^-$ (resp., $\mathcal R_e^+$) the set of all
subgraphs $H \subset G$ such that $e \notin H$ (resp., $e \in H$). Let $H
\in \mathcal R_e^-$ be acyclic. Since $G \setminus \{e\}$ is not strongly
connected and contains one edge less than $G$, one has by Clause
\ref{SSec:NotSSC} above
 \begin{equation}\label{Eq:SumNoE}
\SumSub(\alpha,\mathcal R_e^-) = \SumSub(\alpha,G \setminus \{e\}) = 0.
 \end{equation}

Let now $H \in \mathcal R_e^+$ be acyclic; such $H$ contains no directed
paths joining $b$ with $a$. Since $G \setminus \{e\}$ is not strongly
connected, $G \setminus \{e\}$ does not contain a directed path joining $a$
with $b$ either. It means that such path in $H$ will necessarily contain
$e$, and therefore the graph $H/e \subset G/e$ (obtained by contraction of
the edge $e$) is acyclic. The converse is true for any $e$: if $e \in H$
and $H/e \subset G/e$ is acyclic then $H \subset G$ is acyclic, too. The
graph $G/e$ is strongly connected, contains one edge less (and one vertex
less) than $G$, and $\beta_1(G/e) = \beta_1(G)$, so $\sigma(G/e) =
\sigma(G)$. The graph $H/e$ contains one edge less than $H$, so
$\alpha(H/e) = -\alpha(H)$. Now by the induction hypothesis
 \begin{equation*}
\SumSub(\alpha,R_e^+) = -\SumSub(\alpha,G/e) = -(-1)^{k-1} \sigma(G/e) =
(-1)^k \sigma(G),
 \end{equation*}
and then \eqref{Eq:SumNoE} implies
 \begin{equation*}
\SumSub(\alpha,G) = \SumSub(\alpha,R_e^-) + \SumSub(\alpha,R_e^+) = 0 + (-1)^k
\sigma(G) = (-1)^k \sigma(G).
 \end{equation*}

\subsubsection{$G$ is strongly connected and contains no crucial edges.}

Let $e = [ab]\in G$ be an edge and not a loop: $b \ne a$. Recall that
$G_e^{\vee}$ will denote a graph obtained from $G$ by reversal of the edge
$e$. Since $e$ is not a crucial edge, $G \setminus \{e\} = G_e^{\vee}
\setminus \{e\}$ is strongly connected. So $G_e^{\vee}$ is strongly
connected, too, implying $\sigma(G_e^{\vee}) = \sigma(G)$.

 \begin{lemma} \label{Lm:Reversal}
If the graph $G$ is strongly connected and $e = [ab] \in G$ is not a crucial
edge then $\SumSub(\alpha,G) = \sigma(G)$ if and only if
$\SumSub(\alpha,G_e^{\vee}) = \sigma(G_e^{\vee}) = \sigma(G)$.
 \end{lemma}

 \begin{proof}
Acyclic subgraphs $H \subset G$ are split into five classes:

\newcommand{\Clref}[1]{\mathrm{\ref{#1}}}

{\def \theenumi {\Roman{enumi}}
\def \labelenumi {\theenumi.}
 \begin{enumerate}
\item\label{It:NoEAtoB} $e \notin H$, but $a \succeq b$ in $H$ (that is,
$H$ contains a directed path joining $a$ with $b$).

\item\label{It:NoEBtoA} $e \notin H$, but $b \succeq a$ in $H$.

\item\label{It:NoENoPath} $e \notin H$, and both $a \not\succeq b$ and $b
\not\succeq a$ in $H$.

\item\label{It:EAtoB} $e \in H$, and $a \succeq b$ in $H \setminus \{e\}$.

\item\label{It:ENoPath} $e \in H$, and $a \not\succeq b$ in $H \setminus
\{e\}$.
 \end{enumerate}
}

Obviously, $H \in \Clref{It:NoEAtoB}$ if and only if $H \cup \{e\} \in
\Clref{It:EAtoB}$. The number of edges of $H \cup \{e\}$ is the number of
edges of $H$ plus $1$, so
 \begin{equation}\label{Eq:1Plus4}
\SumSub(\alpha, \Clref{It:NoEAtoB} \cup \Clref{It:EAtoB}) = \sum_{H \in
\mathrm{\ref{It:NoEAtoB}}} (-1)^{\#\text{ of edges of $H$}} \,(1-1) = 0.
 \end{equation}

Also, $H \in \mathrm{\ref{It:NoENoPath}}$ if and only if $H \cup \{e\} \in
\mathrm{\ref{It:ENoPath}}$, and similar to \eqref{Eq:1Plus4} one has
$\SumSub(\alpha,\Clref{It:NoENoPath} \cup \Clref{It:ENoPath}) = 0$, and
therefore
 \begin{equation}\label{Eq:AllEq2}
\SumSub(\alpha,G) = \SumSub(\alpha,\Clref{It:NoEAtoB} \cup
\Clref{It:NoEBtoA} \cup \Clref{It:NoENoPath} \cup \Clref{It:EAtoB} \cup
\Clref{It:ENoPath}) = \SumSub(\alpha,\Clref{It:NoEBtoA}).
 \end{equation}

Like in Clause \ref{SSec:SSCCrucial} if $H \in \Clref{It:ENoPath}$ then
$H/e \subset G/e$ is acyclic, and vice versa, if $e \in H$ and $H/e \subset
G/e$ is acyclic then $H \in \Clref{It:ENoPath}$. The graph $G/e$ is
strongly connected, so by the induction hypothesis
$\SumSub(\alpha,\Clref{It:ENoPath}) = -\SumSub(\alpha,G/e) = -(-1)^{k-1}
\sigma(G/e) = (-1)^k \sigma(G)$, hence $\SumSub(\alpha,\Clref{It:NoENoPath})
= -(-1)^k\sigma(G)$.

If $e \notin H$ and $H$ is acyclic, then $H$ is an acyclic subgraph of the
strongly connected graph $G \setminus \{e\}$. The graph $G$ is strongly
connected, too, so $e$ enters a cycle, and $\beta_1(G \setminus \{e\}) =
\beta_1(G)-1$, which implies $\sigma(G \setminus \{e\}) = -\sigma(G)$. The
graph $G \setminus \{e\}$ contains $k-1 < k$ edges, so by the induction
hypothesis
 \begin{equation*}
\SumSub(\alpha, \Clref{It:NoEAtoB} \cup \Clref{It:NoEBtoA} \cup
\Clref{It:NoENoPath}) = \SumSub(\alpha,G \setminus \{e\}) = (-1)^{k-1}
\sigma(G \setminus \{e\}) = (-1)^k \sigma(G),
 \end{equation*}
and therefore
 \begin{equation}\label{Eq:Sum12}
\SumSub(\alpha, \Clref{It:NoEAtoB} \cup \Clref{It:NoEBtoA}) = 2(-1)^k
\sigma(G).
 \end{equation}

A subgraph $H \subset G$ of class \ref{It:NoEAtoB} is at the same time a
subgraph $H \subset G_e^{\vee}$ of class \ref{It:NoEBtoA}. So,
\eqref{Eq:AllEq2} applied to $G_e^{\vee}$ gives $\SumSub(\alpha,
\Clref{It:NoEAtoB}) = \SumSub(\alpha, G_e^{\vee})$. If follows now from
\eqref{Eq:AllEq2} and \eqref{Eq:Sum12} that
 \begin{equation*}
\SumSub(\alpha,G) + \SumSub(\alpha,G_e^{\vee}) = 2(-1)^k \sigma(G) =
(-1)^k (\sigma(G) + \sigma(G_e^{\vee})),
 \end{equation*}
which proves the lemma.
 \end{proof}

To complete the proof of Theorem \ref{Th:Direct} let $a$ be a vertex of
$G$, and let $e_1 \DT, e_m$ be the complete list of edges finishing at
$a$. Consider the sequence of graphs $G_0 = G$, $G_1 = G_{e_1}^{\vee}$,
$G_2 = (G_1)_{e_2}^{\vee}$, \dots, $G_m = (G_{m-1})_{e_m}^{\vee}$. The
graphs $G_0$ and $G_1$ are strongly connected; the graph $G_m$ is not,
because $a \not\succeq b$ for any $b \ne a$ in it. Take the maximal $\ell$
such that $G_\ell$ is strongly connected. Since $\ell < m$, the graph
$G_{\ell+1}$ exists and is not strongly connected, and therefore $G_\ell
\setminus \{e_{\ell+1}\} = G_{\ell+1} \setminus \{e_{\ell+1}\}$ is not
strongly connected either. So, the edge $e_{\ell+1}$ is crucial for the
graph $G_\ell$, and by Clause \ref{SSec:SSCCrucial} one has
$\SumSub(\alpha,G_\ell) = (-1)^k \SumSub(G_\ell) = (-1)^k \sigma(G)$. The
graphs $G_0=G \DT, G_\ell$ are strongly connected, so for any $i = 0 \DT,
\ell-1$ the edge $e_{i+1}$ is not crucial for the graph $G_i$. Lemma
\ref{Lm:Reversal} implies now
 \begin{align*}
\SumSub(\alpha,G_{\ell-1}) = (-1)^k \sigma(G_{\ell-1})
&\Longrightarrow \SumSub(\alpha,G_{\ell-2}) = (-1)^k \sigma(G_{\ell-2})\\
&\DT\Longrightarrow \SumSub(\alpha,G) = (-1)^k \sigma(G).
 \end{align*}
Theorem \ref{Th:Direct} is proved.

\subsection{Theorem \ref{Th:Diag} follows from Theorem
\ref{Th:Direct}'}\label{SSec:DiagFromDirect}

Note first that the operation $B_i$, and hence $\Lapl$, preserves the sinks
of the graph: if $\Lapl H = \sum_G x_G G$ and $x_G \ne 0$ then $G$ has the
same sinks as $H$. Therefore if $I = \{i_1 \DT, i_s\}$ then
$\Lapl(\det_{n,k}^I) = \sum_G x_G G$ where all the graphs $G$ in the
right-hand side have the sinks $i_1 \DT, i_s$ and have no loops.

Let $G$ be a graph with sinks $i_1 \DT, i_s$ and without loops, and let
$\Phi \in \SSC_{n,k}^I$ (a strongly semiconnected graph with the isolated
vertices $i_1 \DT, i_s$). Denote by $\widehat{\Phi}$ the graph $\Phi$ with
the loops deleted. A contribution of $\Phi$ into $x_G$ is equal to
$\frac{1}{k!} (-1)^{\beta_0(\Phi) - \#\text{ of loops in $\Phi$} + n}$ if
$\widehat{\Phi} \subset G$ and is $0$ otherwise.

The number of edges of $\widehat{\Phi}$ is $k - \#\text{ of loops of
$\Phi$}$. The graph $\widehat{\Phi}$ is strongly semiconnected if and only
if $\Phi$ is. The Euler characteristics of $\widehat{\Phi}$ is
 \begin{equation*}
\beta_0(\widehat{\Phi}) - \beta_1(\widehat{\Phi}) = n - \#\text{ of edges
of $\widehat{\Phi}$} = n-k + \#\text{ of loops of $\Phi$}
 \end{equation*}
and $\beta_0(\widehat{\Phi}) = \beta_0(\Phi)$. Therefore, the contribution
of $\Phi$ into $x_G$ is
 \begin{equation*}
(-1)^{n+k+\beta_0(\widehat{\Phi})+\#\text{ of edges of
$\widehat{\Phi}$}}\frac{1}{k!} =
(-1)^{k+\beta_1(\widehat{\Phi})}\frac{1}{k!}
 \end{equation*}
if $\widehat{\Phi} \subset G$ is strongly semiconnected and $0$ otherwise.
Summing up,
 \begin{equation*}
x_G = \frac{(-1)^k}{k!}\SumSub(\sigma;G) = \frac{1}{k!} \alpha(G)
 \end{equation*}
by Theorem \ref{Th:Direct}'. This proves Theorem \ref{Th:Diag}.

\subsection{Proof of Theorem \ref{Th:Codim1}}\label{SSec:PrCodim1}

Note that $\det_{n,k}^{i/i} = \det_{n,k}^\emptyset + \det_{n,k}^{\{i\}}$.
Applying the operator $\Lapl$ to equation \eqref{Eq:DetViaMinors} and using
Theorem \ref{Th:Diag} with $I = \emptyset$ and $I = \{i\}$ one obtains
 \begin{align*}
0 &= \sum_{i,j=1}^n \Lapl(([ij]) * \det_{n,k}^{i/j}) = \sum_{i=1}^n
\Lapl ([ii]) * \Lapl(\det_{n,k}^\emptyset + \det_{n,k}^{\{i\}}) +
\sum_{\substack{i,j=1\\i \ne j}}^n ([ij]) * \Lapl(\det_{n,k}^{i/j}) \\
&= \sum_{\substack{i,j=1\\i \ne j}}^n ([ij]) * (\Lapl(\det_{n,k}^{i/j}) -
\Lapl(\det_{n,k}^{\{i\}})) = \sum_{\substack{i,j=1\\i \ne j}}^n ([ij]) *
(\Lapl(\det_{n,k}^{i/j}) - \frac{(-1)^k}{k!} U(\AC_{n,k}^{\{i\}})).
 \end{align*}
The $(i,j)$-th term of the identity above consists of graphs where the
edge $[ij]$ carries the number $1$. Therefore different terms of the
identity cannot cancel, so every single term is equal to $0$.

\end{document}